\documentclass{amsart}
\usepackage{amsfonts}
\usepackage{amssymb}
\usepackage{amsmath}
\usepackage{amsthm}
\usepackage[usenames]{color}
\usepackage{amscd}
\usepackage{verbatim}
\usepackage{graphicx}
\usepackage{subfig}		
\usepackage{url}
\usepackage{float}		
\usepackage[colorlinks=true, citecolor=blue, linktocpage]{hyperref}
\usepackage[vcentermath]{youngtab}	
\usepackage{enumerate}	

\usepackage{ytableau}

\newlength{\tabwidth}
\newlength{\tabheight}
\newlength{\tabrule}
\newlength{\tabwidthx}
\newlength{\tabheightx}

\def\gentabbox#1#2#3#4{\vbox to \tabheight{\setlength{\tabrule}{#3}%
  \setlength{\tabwidthx}{#1\tabwidth}\addtolength{\tabwidthx}{\tabrule}%

\setlength{\tabheightx}{#2\tabheight}\addtolength{\tabheightx}{-\tabheight}%
  \hbox to #1\tabwidth{%
 \hspace{-0.5\tabrule}\rule{\tabrule}{#2\tabheight}\hspace{-\tabrule}%
    \vbox to #2\tabheight{\hsize=\tabwidthx%
      \vspace{-0.5\tabrule}\hrule width\tabwidthx height\tabrule%
      \vspace{-0.5\tabrule}\vfil%
      \hbox to \tabwidthx{\hss#4\hss}%
        \vfil\vspace{-0.5\tabrule}%
      \hrule width\tabwidthx height\tabrule\vspace{-0.5\tabrule}}%
 \hspace{-\tabrule}\rule{\tabrule}{#2\tabheight}\hspace{-0.5\tabrule}}%
  \vspace{-\tabheightx}}}
\def\genblankbox#1#2{\vbox to \tabheight{\vfil\hbox to
#1\tabwidth{\hfil}}}
\def\tabbox#1#2#3{\gentabbox{#1}{#2}{0.4pt}{\strut #3}}

\catcode`\:=13 \catcode`\.=13 \catcode`\;=13
\catcode`\>=13 \catcode`\^=13
\def:#1\\{\hbox{$#1$}}
\def.#1{\tabbox{1}{1}{$#1$}}
\def>#1{\tabbox{2}{1}{$#1$}}
\def^#1{\tabbox{1}{2}{$#1$}}
\def;{\genblankbox{1}{1}\relax}
\catcode`\:=12 \catcode`\.=12 \catcode`\;=12
\catcode`\>=12 \catcode`\^=7

\newenvironment{tableau}{\bgroup\catcode`\:=13 \catcode`\.=13
  \catcode`\;=13 \catcode`\>=13 \catcode`\^=13
  \setlength{\tabheight}{3ex}\setlength{\tabwidth}{3ex}%
  \def\b##1##2##3{\gentabbox{##1}{##2}{1.2pt}{\vbox{##3}}}%
  \def\n##1##2##3{\gentabbox{##1}{##2}{0.4pt}{\vbox{##3}}}%
  \vbox\bgroup\offinterlineskip}{\egroup\egroup}

\newtheorem{theorem}{Theorem}[section]

\newtheorem{lemma}[theorem]{Lemma}
\newtheorem{proposition}[theorem]{Proposition}

\newtheorem*{theorem*}{Theorem}

\theoremstyle{definition}
\newtheorem{definition}[theorem]{Definition}

\theoremstyle{remark}

\newtheorem{example}[theorem]{Example}
\newtheorem{remark}[theorem]{Remark}

\DeclareMathOperator{\Sym}{\mathit{S_n}}	

\newcommand{\RSnew}{\mathbf{RS}}				

\DeclareMathOperator{\sh}{sh}			

\newcommand{\Ptab}{\mathbf{P}}						
\newcommand{\Qtab}{\mathbf{Q}}						

\DeclareMathOperator{\sgn}{\mathit{sgn}}		
\DeclareMathOperator{\sign}{\mathit{sign}}	
\DeclareMathOperator{\Inv}{\mathit{Inv}}		
\DeclareMathOperator{\inv}{\mathit{inv}}		
\DeclareMathOperator{\spin}{\mathit{spin}}

\numberwithin{figure}{section}
\numberwithin{theorem}{section}

\begin{document}
\title{On the sign representations for the complex reflection groups $G(r,p,n)$}
\author{Aba Mbirika}
\address{Department of Mathematics, Bowdoin College}
\email{ambirika@bowdoin.edu}

\author{Thomas Pietraho}
\address{Department of Mathematics, Bowdoin College}
\email{tpietrah@bowdoin.edu}

\author{William Silver}
\address{Department of Computer Science, Bowdoin College, and Cognex Corporation}
\email{wsilver@bowdoin.edu}

\date{}
\thanks{The second author would like to thank Skidmore College for its hospitality during the completion of this work.}

\begin{abstract}
We present a formula for the values of the sign representations of the complex reflection groups $G(r,p,n)$ in terms of its image under a generalized Robinson-Schensted algorithm.
\end{abstract}

\maketitle

\pagestyle{myheadings}

\markboth{Aba Mbirika, Thomas Pietraho, and William Silver}{On the sign representations for the complex reflection groups $G(r,p,n)$}


\section{Introduction}

The classical Robinson-Schensted algorithm establishes a bijection between permutations $w \in \Sym$ and ordered pairs of same-shape standard Young tableaux of size $n$.  This map has proven particularly well-suited to certain questions in the representation theory of both $\Sym$ and the semisimple Lie groups of type $A$.  For instance, Kazhdan-Lusztig cells as well as the primitive spectra of semisimple Lie algebras can be readily described in terms of images of this correspondence.

Other sometimes more elementary representation-theoretic information requires more work to extract from standard Young tableaux.  For instance, in independent work, A.~Reifegerste~\cite{reifegerste} and J.~Sj\"ostrand~\cite{sjostrand} developed a method for reading the value of the sign representation of a permutation $w \in \Sym$ based on two tableaux statistics.
Let $w \in \Sym$ and write $RS(w) = (P,Q)$ for its image under the classical Robinson-Schensted map.  If we write $e$ for the number of squares in the even-indexed rows of $P$, let $\sign(T)$ be the sign of a tableau $T$ derived from its inversion number, and let $\sgn$ be the usual sign representation on $\Sym$, then
$$ \sgn(w) = (-1)^{e} \cdot \sign(P) \cdot \sign(Q).$$

The focus of our paper is to extend this result to the complex reflection groups $G(r,p,n).$  Its two main ingredients generalize readily to this setting.
First, the Robinson-Schensted algorithm admits a straightforward extension mapping each $w \in G(r,p,n)$ to a same-shape pair of $r$-multitableaux, see R.~Stanley~\cite{stanley:some} and L.~Iancu~\cite{iancu}.  At the same time, the sign of a permutation in $\Sym$ extends to a family of $r$ one-dimensional representations of $G(r,p,n)$.  After defining new spin and sign statistics on $r$-multitableaux, we prove the following extension of the result of A.~Reifegerste and J.~Sj\"ostrand:

\begin{theorem*}
Let $w \in G(r,p,n)$ and write $\mathbf{RS}(w) = (\mathbf{P},\mathbf{Q})$ for its image under the generalized Robinson-Schensted map.  Given a primitive $r^{th}$ root of unity $\zeta$ and the associated family $\{\sgn_i\}_{i=0}^{r-1}$ of representations of $G(r,p,n)$,  we have
$$\sgn_i(w) = (-1)^{e(\Ptab)} \cdot (\zeta^i)^{\spin(\mathbf{P})+\spin(\mathbf{Q})} \cdot \sign(\mathbf{P}) \cdot \sign(\mathbf{Q}),$$
where $e(\Ptab)$ is the total sum of the lengths of the even-indexed rows of the component tableaux of $\Ptab$.
\end{theorem*}

To prove this, we define a particular set of equivalence classes on $G(r,1,n)$ and show that the formula either holds or fails for all members of a given class.  By construction, each class contains a representative for which the theorem is easy to verify by a direct appeal to the original symmetric group formula, thereby completing the proof.

For the classical Weyl groups, all of which appear among the above complex reflection groups, the M\"obius function of the Bruhat order can be expressed in terms of the values of the sign representation~\cite{verma:mobius}. Its use is ubiquitous in Kazhdan-Lusztig theory, and the above formulas allow the values of the M\"obius function to be read off from the images of the appropriate Robinson-Schensted map.
This is especially relevant in light of the tableaux classification of Kazhdan-Lusztig cells.  In type $A$, left Kazhdan-Lusztig cells consist of those permutations whose recording tableaux agree, see A.~Joseph~\cite{joseph1} or S.~Ariki~\cite{ariki}.  A similar result holds for the so-called asymptotic left cells in the Iwahori-Hecke algebras in type $B$, this time in terms of $2$-multitableaux~\cite{bonnafe:iancu}.


\section{Preliminaries}

After defining the family of complex reflection groups $G(r,p,n)$ and describing their one-dimensional  representations, we define multipartitions, a generalization of the Robinson-Schensted algorithm, and tableaux statistics that we will use to describe these representations.
\subsection{Sign representations}
Consider positive integers $r$, $p$, and $n$ with $p$ dividing $r$ and let $\zeta= \text{exp}(2 \pi \sqrt{-1})/ r)$. We define the complex reflection groups $G(r,p,n)$ as subgroups of $GL_n(\mathbb{C})$ consisting of matrices such that
\begin{itemize}
    \item the entries are either $0$ or powers of $\zeta$,
    \item there is exactly one non-zero entry in each row and column,
    \item the $(r/p)$-th power of the product of all non-zero entries is $1$.
\end{itemize}

Together with thirty-four exceptional groups, the groups $G(r,p,n)$ account for all finite groups generated by complex reflections~\cite{shephard-todd}, and include among them all the classical Weyl groups. In our work the parameter $r$ will generally be fixed allowing us to write simply $W_n$ for the group $G(r,1,n).$  In order to establish  succinct notation, we will write $$[\zeta^{a_1} \sigma_1, \zeta^{a_2} \sigma_2, \ldots, \zeta^{a_n} \sigma_n]$$ for the matrix whose non-zero entry in the $i$th column is $\zeta^{a_i}$ and appears in row $\sigma_i$.  Utilizing this notation, define the set $S=\{s_0, \ldots, s_{n-1}\}$ where
\begin{align*}
s_0 & = [\zeta \cdot 1, 2, 3, \ldots, n] \text{, and } \\
s_i & =[1,2, \ldots, i-1, i+1, i, i+2, \ldots, n].
\end{align*}
Further, let $S'= \{s_0^p, s_0 s_1 s_0, s_i \; | \; 1 \leq i \leq n-1\}.$
The set $S$ generates $W_n$ with presentation given as
$$W_n = \langle s_i \, | \, s_0^r, s_i^2, (s_j s_k)^2, (s_0 s_1)^4, (s_l s_{l+1})^3, \text{ $i \geq 1$, $|j-k|>1$, $l \in [1, n-2]$} \rangle.$$
Subject to similar relations, $S'$ generates a subgroup $G(r,p,n)$ of $W_n$ of index $p$, see S.~Ariki~\cite{ariki95}.

There are exactly $2r$ one-dimensional representations of $W_n$; they divide naturally into two families.
\begin{definition}\label{def:one-diml-representations}
For each integer $i$  between $0$ and $r-1$, we define representations $\sigma_i$ and $\sgn_i$ of $W_n$ by specifying their values on the generating set $S$.  Let

        $$\tau_i^\epsilon(s_j) = \left\{
            \begin{array}{ll}
                \zeta^{i} & \text{ if $j=0$, and} \\
                (-1)^\epsilon & \text{ if $j=1, \ldots, n-1$}
            \end{array}
                \right.
        $$
and define $\sigma_i = \tau_i^0$ and $\sgn_i = \tau_i^1$.  Each becomes a representation of the subgroup $G(r,p,n)$ by restriction.
\end{definition}

\subsection{Multitableaux} \label{subsection:multitableaux}

We write a partition  $\lambda$ of an integer $m$ as a nonincreasing sequence of positive integers $(\lambda_1, \lambda_2, \ldots, \lambda_k)$ and define its rank as $|\lambda|=m$.  A {\it Young diagram} $[\lambda]$ of $\lambda$ is a left-justified array of boxes containing $\lambda_i$ boxes in its $i$th row.  The shape of a Young diagram will refer to its underlying partition.
With the integer $r$ fixed, a {\it multipartition of rank $n$} is an $r$-tuple $$\boldsymbol{\lambda} = (\lambda^0, \lambda^1, \ldots, \lambda^{r-1})$$ of partitions the sum of whose individual ranks equals $n$.  The {\it Young diagram} $[\boldsymbol{\lambda}]$ of $\boldsymbol{\lambda}$ is the $r$-tuple $([\lambda^0], \ldots, [\lambda^{r-1}])$.  We refer to $\boldsymbol{\lambda}$ as the {\it shape} of the diagram $[\boldsymbol{\lambda}]$ and define $|\boldsymbol{\lambda}|=n$.   We will follow a convention of denoting objects derived from multipartitions in boldface while writing those derived from single partitions using a normal weight font.

A {\it standard Young tableaux of shape} $\boldsymbol{\lambda}$  is the Young diagram $[\boldsymbol{\lambda}]$ of rank $n$ together with a labeling of each of its boxes of with the elements of $\mathbb{N}_n$ in such a way that each number is used exactly once, and the labels of the boxes within each component Young diagram $[\lambda^i]$ increase along its rows and down its columns.  Remembering that $r$ is fixed, we will write $\mathbf{SYT}_n$ for the set of all standard Young tableaux of rank $n$ whose shape is a multipartition with $r$ components.

\begin{definition}  A tableau $\mathbf{T} = (T_1, T_2, \ldots, T_r) \in \mathbf{SYT}_n$ will be called {\it ascending} if for each $i<r$, the elements of the set of labels of the component tableau $T_i$ are pairwise smaller than the elements of the set of labels of $T_{i+1}$.
\end{definition}

\begin{example}
Take $r=3$.  The following standard Young tableau $\mathbf{T}$ has shape $\boldsymbol{\lambda}=((2,1), (1,1), (3,3))$, is of rank $11$, and is ascending:
$$
\raisebox{.07in}{$\mathbf{T}=\Big($}
\hspace{.05in}
\begin{tiny}
\begin{tableau}
    :.{1}.{3}\\
    :.{2}\\
\end{tableau}
,\hspace{.1in}
\begin{tableau}
    :.{4}\\
    :.{5}\\
\end{tableau}
\hspace{.02in},\hspace{.1in}
\begin{tableau}
    :.{6}.{7}.{10}\\
    :.{8}.{9}.{11}\\
\end{tableau}
\end{tiny}
\hspace{.05in}
\raisebox{.07in}{\Big)}.
$$
\end{example}

 Following~\cite{stanley:some} and ~\cite{iancu}, we  define a map from $W_n$ to same-shape pairs of $r$-tuples of standard Young tableaux.
Consider $w = [\zeta^{a_1} \sigma_1, \zeta^{a_2} \sigma_2, \ldots, \zeta^{a_n} \sigma_n] \in W_n$ and define ordered sets $w^{(k)}= (\sigma_i \, | \, a_i = k)$ for $0 \leq k < r$.  Let $RS(w^{(k)}) = (P_k, Q_k)$ be the image of the sequence $w^{(k)}$ under the usual Robinson-Schensted map, labeling squares of $Q_k$ according to the relative positions of $i \in w^{(k)}$ within $w$, and define
$$\mathbf{P}(w) = (P_0, P_1, \ldots, P_{r-1}) \text{ \phantom{XX} and \phantom{XX}} \mathbf{Q}(w) = (Q_0, Q_1, \ldots, Q_{r-1}).$$
The multitableaux Robinson-Schensted map is defined by $\mathbf{RS}(w) = (\mathbf{P}(w), \mathbf{Q}(w)).$
It maps $W_n$ onto the set of same-shape pairs of elements of $\mathbf{SYT}_n$ and is in fact a bijection.

\subsection{Tableaux and multitableaux statistics}

Our goal is to describe values of the sign representations on $W_n$ under the above generalization of the Robinson-Schensted map.  To do so, we rely on a few statistics that can be readily computed from multitableaux.

\begin{definition}
An {\it inversion} in a Young tableau $T$ is a pair $(i,j)$ with $j>i$ for which the box labeled by $i$ is contained in a row strictly below the box labeled $j$.  Let $\Inv(T)$ be the set of inversions in $T$, and write $\inv(T)$ for its cardinality. If $\mathbf{T}= (T_0, T_1, \ldots T_{r-1})$ is a multitableau, we extend this notion and define:
$$\Inv(\mathbf{T}) = \bigsqcup_k \Inv (T_k) \sqcup \bigsqcup_{k<l} \Inv(T_k,T_l)$$
where $\Inv(T_k,T_l) = \{(j,i) \, | \, j>i, j \text{ is a label in $T_k$, $i$ is a label in $T_l$}\}.$  We will be mainly interested in the parity of the size of this set and define $$\sign(\mathbf{T}) = (-1)^{\inv(\mathbf{T})}.$$
\end{definition}

\begin{definition}
For a Young tableau $T$, write $e(T)$ for the total number of boxes in its rows of even index.  For a multitableau $\mathbf{T}= (T_0, T_1, \ldots T_{r-1})$, we write $\sh(T_k)$ for the shape of the Young diagram underlying $T_k$ and  define the statistics $e$ and $\spin$ as follows:
$$e(\mathbf{T}) = \sum_{k=0}^{r-1} e(T_k)  \; \text{ and } \; \spin(\mathbf{T}) = \frac{1}{2}  \sum_{k=0}^{r-1} k \cdot |\sh(T_k)|.$$
\end{definition}

The $\spin$ statistic provides a simple description of the image of the subgroup $G(r,p,n)$ under the $r$-multitableaux Robinson-Schensted map.

\begin{proposition}
$(\mathbf{P}, \mathbf{Q}) \in \mathbf{RS}(G(r,p,n))$ if and only if
$2 \spin(\mathbf{P}) \equiv 0 \pmod{p}.$
\end{proposition}

\begin{proof}

Write $w= [\zeta^{a_1} \sigma_1, \zeta^{a_2} \sigma_2, \ldots, \zeta^{a_n} \sigma_n ] \in G(r,p,n)$ and for each $0 \leq k \leq r-1$, set $\mathcal{A}_k := \{ i \, | \, a_i = k \}$.  If $\mathbf{RS}(w)=(\mathbf{P},\mathbf{Q})$, then $|\sh(P_k)| = |\mathcal{A}_k|$, and hence it follows that $k \cdot |\sh(P_k)|$ equals $\sum_{i \in \mathcal{A}_k} a_i$.  Since $\mathbb{N}_n = \bigsqcup_{k=0}^{r-1} \mathcal{A}_k$, we conclude that $$ \sum_{k=1}^{n} a_k = \sum_{k=0}^{r-1} k \cdot |\sh(P_k)| = 2 \spin(\mathbf{P}).$$  Finally, since $w \in G(r,p,n)$ if and only if $(\zeta^{a_1}\zeta^{a_2}\cdots\zeta^{a_n})^{\frac{r}{p}} = 1$, it follows that
$w \in G(r,p,n)$ if and only if  $2 \spin(\mathbf{P}) \equiv 0 \pmod{p},$ as claimed.
\end{proof}

\vspace{.1in}

\subsection{A set of functions and an example}
We define a sequence of functions on $W_n$.   In the next section we will show that they coincide with the sign representations on $W_n$.  Again, for $w \in W_n$, let $\mathbf{RS}(w)=(\mathbf{P},\mathbf{Q})$. For $0 \leq i < r$, we will write
$$\pi_i(w) = (-1)^{e(\Ptab)} \cdot (\zeta^{i})^{\spin(\mathbf{P}) + \spin(\mathbf{Q})} \cdot \sign(\mathbf{P}) \cdot \sign(\mathbf{Q}).$$

\vspace{.1in}

\begin{example}\label{running_example}
Consider $w = [ \zeta^1 \, 5,  1, \zeta^2 \, 3, 6, \zeta^2 \, 7, \zeta^1 \, 4,  2, 8]$ in $G(4,1,8).$  Recalling the notation in Section \ref{subsection:multitableaux}, we have
$w^{(0)} = (1, 6, 2, 8)$,
$w^{(1)} = (5,4)$,
$w^{(2)} = (3,7)$, and
$w^{(3)} = \emptyset$.  Further,
\begin{align*}
 & RS(w^{(0)}) =  \begin{tiny} \left(\; \young(128,6)\hspace{.05in} ,\hspace{.05in} \young(248,7) \;\right) \end{tiny}
& RS(w^{(1)}) & =  \begin{tiny} \left(\; \young(4,5)\hspace{.05in} , \hspace{.05in} \young(1,6) \;\right)\end{tiny} \\
 & RS(w^{(2)}) =  \begin{tiny}\left(\; \young(37)\hspace{.05in} , \hspace{.05in}\young(35) \;\right)\end{tiny}
& RS(w^{(3)})  & =  \begin{tiny} \left(\; \emptyset, \emptyset \;\right)\end{tiny}.
\end{align*}
From these we construct the Robinson-Schensted image of $w$:
$$\RSnew(w)=(\Ptab, \Qtab) = \begin{tiny} \left( \; \left( \, \young(128,6) \, , \, \young(4,5) \, , \, \young(37) \, , \, \emptyset \, \right) \;, \;\left( \, \young(248,7) \, , \, \young(1,6) \, , \, \young(35) \, , \, \emptyset \, \right) \; \right).\end{tiny}$$
We read off $\inv(\Ptab) = 10$, $\inv(\Qtab) = 12$, $e(\Ptab) = 2$, and $\spin(\Ptab) = \spin(\Qtab) = 3$.  Hence $\pi_i(w) =(\zeta^i)^2$ which coincides with $\sgn_i(w)$.

\end{example}


\section{Sign under the Robinson-Schensted map}

The aim of this section is to verify the formulas for the one-dimensional representations of $W_n$ given above. We reduce the problem to the setting of the classical Robinson-Schensted map by considering a particular equivalence relation on the elements of $W_n$.

\subsection{Ascending elements and admissible transformations}

Recall the notation from Section \ref{subsection:multitableaux} and
write $w \in W_n$ in one-line notation as $$w=[\zeta^{a_1} \sigma_1, \zeta^{a_2} \sigma_2, \ldots, \zeta^{a_n} \sigma_n ].$$
We will say that $w$ is {\it ascending} if the exponents $a_k$ increase weakly and for every integer $0 \leq i<r-1$ all elements of $w^{(i)}$ are smaller than the elements of $w^{(i+1)}$.
The following fact is immediate:

\begin{proposition}
Suppose that $w \in W_n$ is ascending.  Then its left and right Robinson-Schensted tableaux $\mathbf{P}{(w)}$ and $\mathbf{Q}{(w)}$ are both ascending.
\end{proposition}

We proceed to define equivalence classes on $W_n$, each with an ascending representative.  Adopting the notation from the definition above, consider $i>0$ and define $L_i(w)=s_i \cdot w$ whenever $a_i \neq a_{i+1}$ and  $R_i(w)=w \cdot s_i$ whenever
$\sigma_l =i$, $\sigma_k=i+1$, and $a_l \neq a_{k}$.  We will call operators of this form {\it left} and {\it right admissible}, respectively. To define an equivalence relation, we let $L_i(w) \sim w$ and $R_i(w) \sim w$ for every $i$ and take its transitive closure.

A right admissible operator permutes adjacent entries of $w$, while each left admissible operator exchanges some $\sigma_k$ and $\sigma_l$ which have different exponents in $w$ and satisfy $|\sigma_k-\sigma_l|=1$ while preserving the exponents in each position.  Thus by successive application of right admissible operators, the entries of $w$ can be arranged to have weakly increasing exponents.  Furthermore, by successive application of left admissible operators, the $\sigma_k$ can be permuted so that
for every integer $0 \leq i<r-1$ all  entries  with coefficient $\zeta^i$  are smaller than those with coefficient $\zeta^{i+1}$.  We have shown the following:

\begin{proposition}
Each $w \in W_n$ has an ascending representative $\tilde{w}$ in the equivalence class defined by left and right admissible operators.
\end{proposition}

Admissible operators  behave well with respect to the multitableaux Robinson-Schensted map:

\begin{proposition}
For $w \in W_n$, $\mathbf{P}(w)=\mathbf{P}(R_i(w))$ and $\mathbf{Q}(w)=\mathbf{Q}(L_i(w)).$
\end{proposition}

\begin{proof}
Applying a right admissible transformation $R_i$ to $w$ does not change the relative order of the entries sharing the coefficient $\zeta^{a_k}$.  Thus as ordered sets, we have $w^{(k)} = (R_i(w))^{(k)}$ for all $k$ and the first part of the proposition follows.  
Suppose that $\sigma_k$ and $\sigma_l$ have different exponents in $w$ and satisfy $|\sigma_k-\sigma_l|=1$.  Exchanging the two while preserving the exponents in each position preserves the order new boxes are appended in constructing the component tableaux of $\mathbf{P}(w)$, thus a left admissible operation fixes $\mathbf{Q}(w)$, as claimed.
\end{proof}

The multitableau $\mathbf{Q}(R_i(w))$ is obtained from $\mathbf{Q}(w)$ by exchanging boxes with labels $i$ and $i+1$.  The result is standard as the boxes labeled $i$ and $i+1$ lie in different component tableaux of $\mathbf{Q}(w)$.  
Furthermore,  $\Inv(\mathbf{Q}(R_i(w)))$ equals  $\Inv(\mathbf{Q}(w))$ with $i$ and $i+1$ interchanged and the inversion $(i+1,i)$ either added or removed.
In particular, a right admissible transformation changes  $\inv(\mathbf{Q}(w))$ by one while fixing $\inv(Q_k(w))$ for each $k$.  Similar reasoning can be applied to $\mathbf{P}(L_i(w))$, thereby proving:

\begin{proposition}\label{proposition:inversions} A right admissible transformation changes $\inv(\mathbf{Q}(w))$ by one while fixing $\inv(Q_k(w))$ for each $k$, while a left admissible transformation changes $\inv(\mathbf{P}(w))$ by one while fixing $\inv(P_k(w))$ for each $k$.
\end{proposition}

We are ready to prove the main result of this section which shows that it is only necessary to verify our formula for the
$\sgn_i$ representations in the setting of ascending elements of $W_n$.

\begin{lemma}\label{lemma:ascending}
Consider $w \in W_n$ and let $\tilde{w}$ be an ascending representative of $w$.  Then for $0 \leq i < r$,
\begin{align*}
\pi_i(w) = \sgn_i(w) & \; \mbox{ if and only if } \; \pi_i(\tilde{w}) = \sgn_i(\tilde{w}).
\end{align*}
\end{lemma}

    \begin{proof}
        First note that $\tilde{w} = a \cdot w \cdot b$ for $a,b \in S_n$ and consequently for each $i$,
        $$\sgn_i(\tilde{w})= (-1)^{\ell(a)+ \ell(b)} \cdot \sgn_i(w).$$
        Since $e$ and $\spin$ are shape-based statistics and admissible transformations preserve the shapes of the underlying tableaux, the only possible difference between $\pi_i(w)$ and $\pi_i(\tilde{w})$ lies in the signs of their left and right tableaux.  This is addressed explicitly in Proposition \ref{proposition:inversions}.  We obtain
        $\pi_i(\tilde{w})= (-1)^{\ell(a)+ \ell(b)} \cdot \pi_i(w)$
        and the lemma follows.
    \end{proof}

\subsection{The general sign formula}

\begin{theorem}\label{main_theorem}
Let $w \in W_n$ and write $\RSnew(w) = (\Ptab,\Qtab)$ for its image under the
generalized Robinson-Schensted map.  Given a primitive $r^{th}$ root of unity $\zeta$ and the associated family $\{\sgn_i \}_{i=0}^{r-1}$ of representations of $W_n$, we have
$$\sgn_i(w) = (-1)^{e(\Ptab)} \cdot (\zeta^i)^{\spin(\Ptab) + \spin(\Qtab)} \cdot \sign(\Ptab) \cdot \sign(\Qtab).$$
\end{theorem}

\begin{proof}
By Lemma \ref{lemma:ascending}, it is enough to verify the above formula for ascending elements of $W_n$.  Consequently,  consider an ascending $w \in W_n$, writing it in one-line notation as
$w=[\zeta^{a_1} \sigma_1, \zeta^{a_2} \sigma_2, \ldots, \zeta^{a_n} \sigma_n ].$  Let $\{P_k\}_{k=0}^{r-1}$ and $\{Q_k\}_{k=0}^{r-1}$ be the component tableaux of  $\mathbf{P}$ and $\mathbf{Q}$.  Since the latter tableaux are both ascending and of the same shape, for each $k$, $P_k$ and $Q_k$ are labeled by the same set of integers, say $\mathbf{n}_k$.  Let $S_{\mathbf{n}_k}$ be the group of permutations of $\mathbf{n}_k$ and  define $u_k = {RS}^{-1}(P_k, Q_k) \in S_{\mathbf{n}_k}$ for the natural extension of the classical Robinson-Schensted map to the ordered set $\mathbf{n}_k$.  Let $u_k= 1 \in W_n$ if $\mathbf{n}_k =\emptyset$.  Also let $u_r = [ \zeta^{a_1} 1, \zeta^{a_2} 2, \ldots, \zeta^{a_n} n ] \in W_n$.  Note that we can view $S_{\mathbf{n}_k}$ as a subgroup of $S_n$ for each $k$.  From this perspective,  since $\mathbf{P}$ and $\mathbf{Q}$ are ascending, we have $w = u_0 \cdot u_1 \cdots u_r$ and consequently
\begin{align*}
\sgn_i(w) & = \sgn_i(u_0) \cdot \sgn_i(u_1) \cdots \sgn_i(u_r)\\
& = \sgn(u_0) \cdot \sgn(u_1) \cdots \sgn({u_{r-1}}) \cdot \sgn_i(u_r)
\end{align*}
where $\sgn$ is the usual sign of a permutation.

We would like to reduce this expression to the formula proposed by the theorem.  The first step invokes the original result of A.~Reifegerste~\cite{reifegerste} and J.~Sj\"ostrand~\cite{sjostrand}.  For $k<r$, we have
$\sgn(u_k) = (-1)^{e_k} \cdot \sign(P_k) \cdot \sign(Q_k),$
where $e_k = e(P_k)$. Furthermore, $\sgn_i(u_r)$ is easily computed to be $(\zeta^i)^{\sum_{k=1}^n a_k}$ and consequently we have:
$$\sgn_i(w)=(-1)^{\sum_{k=0}^{r-1} e_k} \cdot (\zeta^i)^{\sum_{k=1}^n a_k} \cdot (-1)^{\sum_{k=0}^{r-1} \inv(P_k) + \inv(Q_k)}.$$
To further simplify the above expression we note that $\sum_{k=0}^{r-1} e_k = e(\mathbf{P})$,  by definition.  Furthermore, if $n_k=|w^{(k)}|$ is the cardinality of the set of entries of $w$ with exponent equal to $k$, then $\sum_{k=1}^n a_k = \sum_{k=0}^{r-1} k \cdot n_k = 2 \spin(\mathbf{P}) = \spin(\mathbf{P}) + \spin(\mathbf{Q}).$

Finally, since both $\mathbf{P}$ and $\mathbf{Q}$ are ascending tableaux, the sets of inversions $\Inv(P_k, P_l)$ and $\Inv(Q_k, Q_l)$ are both empty for $k < l$.  Consequently, $\inv(\mathbf{P})$ and $\inv(\mathbf{Q})$ are just the sums of the numbers of inversions in their component tableaux.  The theorem now holds by direct substitution.
\end{proof}

\begin{example}
Continuing with $w$ from Example \ref{running_example}, we have the ascending representative $\tilde{w} = [1, 3, 2, 4, \zeta^1 \, 6, \zeta^1 \, 5, \zeta^2 \, 7, \zeta^2 \, 8]$.  Its Robinson-Schensted image is:
$$\RSnew(\tilde{w})=(\Ptab, \Qtab) =  \begin{tiny}\left( \; \left( \, \young(124,3) \, , \, \young(5,6) \, , \, \young(78) \, , \, \emptyset \, \right) \;, \;\left( \, \young(124,3) \, , \, \young(5,6) \, , \, \young(78) \, , \, \emptyset \, \right) \; \right).\end{tiny}$$
In contrast to the behavior of the Robinson-Schensted map in the setting of classical Weyl groups, the fact that $\Ptab$ and $\Qtab$
coincide does not necessarily imply that the corresponding group element is an
involution, as  occurs here. We decompose $\tilde{w}$ as $u_0 \cdot u_1 \cdot u_2 \cdot u_3 \cdot u_4$ where $u_3=1$ and:
\begin{align*}
&u_0 = [1,3,2,4] \in S_{\{1,2,3,4\}}
&u_1 &= [6,5] \in S_{\{5,6\}}\\
&u_2 = [7,8] \in S_{\{7,8\}}
&u_4 &= [ 1,  2,  3, 4, \zeta^1 \, 5, \zeta^1 \, 6, \zeta^2 \, 7, \zeta^2 \, 8] \in W_8.
\end{align*}
Since $\sgn(u_0)=\sgn(u_1)=-1$, $\sgn(u_2)=1$, and $\sgn_i(u_4) = (\zeta^i)^2$, we compute $\sgn_i(\tilde{w}) = (\zeta^i)^2$.  Furthermore, since $\inv(\Ptab)=\inv(\Qtab)=1$, $e(\Ptab)=2$, and $\spin(\Ptab)=\spin(\Qtab)=3$, we compute $\pi_i(\tilde{w}) = (\zeta^i)^2$, which agrees with $\sgn_i(\tilde{w})$, as expected.
\end{example}

\begin{remark}
It is also possible to approach the present problem by examining the coplactic classes on $W_n$ defined by the Robinson-Schensted map instead of the equivalence classes used here.  In lieu of an ascending class representative, one can use the unique element in each coplactic class whose left multitableau has no inversions.  Using the generalized Knuth relations for $W_n$ described in \cite{iancu}, one can then produce the same formula by following their action on $\sgn_i$, albeit with a little more work.   The advantage of the equivalence classes described in this paper is that they reduce the problem to the symmetric group setting as completely as possible.
\end{remark}



\begin{thebibliography}{}

\bibitem{ariki95}
S.~Ariki.
\newblock Representation theory of a {H}ecke algebra of {$G(r,p,n)$}
\newblock {\em J. Algebra}, 77(1):164--185, 1995.

\bibitem{ariki}
S.~Ariki.
\newblock {R}obinson-{S}chensted correspondence and left cells. \textit{Combinatorial methods in representation theory} (Kyoto, 1998),
\newblock {\em Adv. Stud. Pure Math.}, 28:1--20, Kinokuniya, Tokyo, 2000.

\bibitem{bonnafe:iancu}
C.~Bonnaf{\'e} and L.~Iancu.
\newblock Left cells in type $B\sb n$ with unequal parameters.
\newblock {\em Represent. Theory}, 7:587--609, 2003.

\bibitem{iancu}
L.~Iancu.
\newblock Cellules de {K}azhdan-{L}usztig et correspondance de {R}obinson-{S}chensted.
\newblock {\em C. R. Math. Acad. Sci. Paris}, 336(10):791--794, 2003.

\bibitem{joseph1}
A.~Joseph.
\newblock
A characteristic variety for the primitive spectrum of a semisimple {L}ie algebra.  Short version in: \textit{Non-Commutative Harmonic Analysis} (Actes Colloq., Marseille-Luminy, 1976),
\newblock {\em Lecture Notes in Math.}, 587:102--118, Spinger, Berlin, 1977.

\bibitem{reifegerste}
A.~Reifegerste.
\newblock
Permutation sign under the {R}obinson-{S}chensted correspondence.
\newblock {\em Ann. Comb.}, 8(1):103--112, 2004.

\bibitem{shephard-todd}
G. C.~Shephard and T. A.~Todd.
\newblock
Finite unitary reflection groups.
\newblock {\em Canadian J. Math.}, 6:274--304, 1954.

\bibitem{sjostrand}
J.~Sj\"ostrand.
\newblock
On the sign-imbalance of partition shapes.
\newblock {\em J. Combin. Theory Ser. A}, 111(2):190--203, 2005.

\bibitem{stanley:some}
R.~Stanley.
\newblock
Some aspects of groups acting on finite posets.
\newblock
{\em J. Combin. Theory Ser. A}, 32(2):132--161, 1982.

\bibitem{verma:mobius}
D.-N.~Verma.
\newblock
M\"obius inversion for the {B}ruhat ordering on a {W}eyl group.
\newblock
{\em Ann. Sci. \'Ecole Norm. Sup.}, 4(4):393--398, 1971.

\end{thebibliography}
\end{document}